\documentclass[11pt,english]{smfart}
\usepackage[T1]{fontenc}
\usepackage[english,francais]{babel}
\usepackage{graphicx}
\usepackage{hyperref}
\usepackage{amssymb}
\usepackage{mathtools}
\usepackage{tikz}
\usepackage{pb-diagram}
\usepackage{amsmath}
\usetikzlibrary{matrix}
\usepackage{smfthm}
\usepackage{mathrsfs}
\usepackage{pb-diagram}
\usepackage{yfonts}
\usepackage[inline]{enumitem}
\usepackage{color}
\usepackage{verbatim}
\usepackage{bbm}
\usetikzlibrary{matrix}
\theoremstyle{plain}

\def\R{\mathbb{R}}
\def\Q{\mathbb{Q}}

\def\Z{\mathbb{Z}}
\def\N{\mathbb{N}}
\def\O{\mathcal{O}}
\def\L{\Lambda}

\def\nm{\lVert\cdot\rVert}

\def\deg{\widehat{\mathrm{deg}}}
\def\vol{\widehat{\mathrm{vol}}}
\def\adlcv{(K,(\Omega, \mathcal A,\nu),\phi)}

\def\Div{\mathrm{Div}}
\def\mumin{\widehat{\mu}_{\min}}
\def\mumax{\widehat\mu_{\max}}

\def\xanomega{X^{\mathrm{an}}_\omega}

\makeatletter
\newcommand\tint{\mathop{\mathpalette\tb@int{t}}\!\int}
\newcommand\bint{\mathop{\mathpalette\tb@int{b}}\!\int}
\newcommand\tb@int[2]{%
  \sbox\z@{$\m@th#1\int$}%
  \if#2t%
    \rlap{\hbox to\wd\z@{%
      \hfil
      \vrule width .35em height \dimexpr\ht\z@+1.4pt\relax depth -\dimexpr\ht\z@+1pt\relax
      \kern.05em 
    }}
  \else
    \rlap{\hbox to\wd\z@{%
      \vrule width .35em height -\dimexpr\dp\z@+1pt\relax depth \dimexpr\dp\z@+1.4pt\relax
      \hfil
    }}
  \fi
}
\makeatother
\title{Arithmetic over trivially valued field and its applications}
\author{Wenbin LUO}
\date{June 2020}

\begin{document}

\maketitle
\begin{abstract}
By some result on the study of arithemtic over trivially valued field, we find its applications to Arakelov geometry over adelic curves. We prove a partial result of the continuity of arithmetic $\chi$-volume along semiample divisors. Moreover, we give a upper bound estimate of arithmetic Hilbert-Samuel function.
\end{abstract}
\tableofcontents
\section{Introduction}
Arakelov geometry is a theory to study varieties over $\mathcal O_K$ where $K$ is a number field. Since the closed points of $\mathrm{Spec}\mathcal O_K$ only give rise to non-Archimedean places of $K$, the main idea of Arakelov Geometry is to "compactify" $\mathrm{Spec}\mathcal O_K$ by adding Archimedean places, and the corresponding fibers are nothing but analytification of the generic fiber. In order to generalize the theory to the case over a function field or even more general cases, Chen and Moriwaki established the theory of Arakelov geometry over adelic curves\cite{adelic}. Here an adelic curve is a field equipped with a set of absolute values parametrised by a measure space. Such a structure can be easily constructed for any global fields by using the theory of adèles. The first step is to give a theory of geometry of numbers. Inspired by the study of semistability of vector bundles over projective regular curves, the arithmetic slope theory was established, and the tensorial minimal slope property was firstly proved in \cite{BOST2013}. It's easy to find that some slope-related results can be roughly described by arithmetic over trivially valued field, especially the analogous Harder-Narasimhan filtration of a adelic vector bundle due to the obvious fact that any $\R$-filtration induces an ultrametric norm over trivially valued field.

This article is to dig deeper on this, and give mainly two applications. The first is the application to the study of $\chi$-volume function $\vol_\chi(\cdot)$, which is an analogy of Euler characteristic. We can easily show some boundedness of $\vol_\chi(\cdot)$ with the assumption on the finite generation of section ring. Therefore we obtain the following result on the continuity of $\vol_\chi(\cdot)$.
\begin{theo}
Let $\overline D=(D,g)$, $\overline E_1=(E_1,h_1)$, $\dots, \overline E_r=(E_r,h_r)$ be adelic $\Q$-Cartier divisors on $X$ such that $D$ and $E_i$'s are semiample.
\begin{equation}\lim_{\substack{\epsilon_1+\cdots+\epsilon_r\rightarrow 0\\\epsilon_i\in\Q_{\geq 0}}}\vol_\chi(\overline D+\sum_{i=1}^r\epsilon_i \overline E_i)=\vol_\chi(\overline D)\end{equation}
\end{theo}
The second result worth to mention is about arithmetic Hilbert-Samuel function. In \cite{chen:hal}, H. Chen give an upper bound for arithmetic Hilbert-Samuel function of big arithmetic varieties with an assumption on the relationship between slopes and successive minima. Here in the case of $\mathrm{char}K=0$, we use the arithmetic over trivially valued to make a modification such that we can remove the assumption.
\begin{theo}[Upper bound of arithmetic Hilbert-Samuel function] Let $(D,g)$ be an adelic divisor on $X$ with $D$ being big. For each $n\in \N$, let $\overline E_n$ denote the pair of $E_n=H^0(X,\O_X(nD))$ and norm family $\xi_{ng}$. Then it holds that
$$\deg_+(\overline E_n)\leqslant \vol(D,g)\frac{n^{d+1}}{(d+1)!}+O(n^d)+(C+1/2)\cdot r_n \ln(r_n).$$
Moreover, if $D$ is ample, then 
\begin{equation}\deg(\overline E_n)\leqslant \vol_\chi(D,g)\frac{n^{d+1}}{(d+1)!}+O(n^d)+(C+1/2)\cdot r_n\ln(r_n)\end{equation}
\end{theo}
\section{Normed vector spaces and graded linear series over trivially valued field}
\subsection{Arithmetic over trivially valued field}
Let $E$ be a finitely dimensional vector space over $K$. Then any $\R$-filtration $\mathcal F^t$ on $E$ induces an ultrametric norm $\nm_{\mathcal F}$ over $(K,\lvert\cdot\rvert_0)$ which is given by
$$\lVert x\rVert_{\mathcal F}:=\exp(-\sup\{t\mid x\in\mathcal F^t E\}).$$
Moreover, this is actually a bijection between the sets of ultrametic norms and $\R$-filtrations. Moreover, the minimal slope corresponding to the $\R$-filtration can be computed by $\mumin(E,\mathcal F^t)=-\ln(\max\{\lVert x\rVert_{\mathcal F}\mid x\in E\})$. Therefore we give the following definition.
\begin{defi}
Let $\overline E=(E,\nm)$ be an ultrametric normed vector space over $(K,\lvert\cdot\rvert_0)$. We define the minimal slope of $\overline E$ by
$$\mumin(\overline E)=\begin{cases}-\ln(\max\{\lVert x\rVert\mid x\in E\}),\text{ if }E\not=0.\\
+\infty,\text{ if }E=0.
\end{cases}$$
In particular, for any injective homomorphism $f:F\rightarrow E$, we denote by $\nm_{sub(f)}$ the induced ultrametric norm on $F$, then
$$\mumin(F,\nm_{sub(f)})\geqslant \mumin(\overline E)$$
\end{defi}
\begin{prop}
Let $\overline E=(E,\nm)$ be an ultrametrically normed vector space over $(K,\lvert\cdot\rvert_0)$.
Let $$0\rightarrow F\xrightarrow{f} E\rightarrow G\rightarrow 0$$ be an exact sequence of vector spaces over $K$. Then it holds that
$$\mumin(\overline E)=\min\{\mumin(F,\nm_{sub(f)}),\mumin(G,\nm_{quot})\}.$$
\end{prop}
\begin{proof}When $E=0$, the proposition is trivial. We thus assume that $E\not=0$
By definition, we can easily see that 
$$\mumin(\overline E)\leqslant\min\{\mumin(F,\nm_{sub(f)}),\mumin(G,\nm_{quot})\}.$$

Let $x\in E$ be an element with $\lVert x\rVert=\exp(-\mumin(\overline E))$.
If $x$ is contained in the image of $F$, then $\mumin(F,\nm_{sub(f)})=\mumin(\overline E)$, we are done.
Otherwise, $\tilde x\not=0\in G$, it holds that $$\lVert \tilde x\rVert_{quot}=\inf\limits_{y\in F}\lVert x+f(y)\rVert=\inf\limits_{y\in F}\max\{\lVert x\rVert, \lVert f(y)\rVert\}=\lVert x\rVert$$
which implies that $\mumin(G,\nm_{quot})=\mumin(\overline E)$.
\end{proof}
\begin{rema}
The above proposition can be also obtained by using the inequality in \cite[Proposition 4.3.32]{adelic} since an ultrametric normed vector space can be viewed as an adelic vector bundle over the trivially valued field $K$.
\end{rema}
\begin{defi}
Let ${\overline E_i=(E_i,\nm_i)}_{i=1}^n$ be a collection of ultrametric normed vector spaces over $(K,\lvert\cdot\rvert_0)$. We define the a norm $\nm$ on $\mathop\oplus\limits_{i=1}^n E_i$ by
$$\lVert (a_1,\cdots,a_n)\rVert:=\max_{i=1}^n\{\lVert a_i\rVert_i\}$$
for $(a_1,\cdots,a_n)\in\mathop\oplus\limits_{i=1}^n E_i$, which is called the direct sum of norms $\{\nm_i\}_{i=1}^n$, and the ultrametric normed vector space $(\mathop\oplus\limits_{i=1}^n E_i,\nm)$ can be denoted by $\mathop\oplus\limits_{i=1,\dots,n}^{\perp}\overline E_i.$
\end{defi}
\begin{rema}
It's easy to see that $$\mumin\left(\mathop\oplus\limits_{i=1,\dots,n}^{\perp}\overline E_i\right)= \min_{i=1}^n\{\mumin(\overline E_i)\}.$$
Moreover, the successive slopes of $\mathop\oplus\limits_{i=1,\dots,n}^{\perp}\overline E_i$ is just the sorted sequence of the union of successive slopes of $\overline E_i$.
\end{rema}
\begin{defi}
Let $\overline E=(E,\nm_E)$ and $\overline F=(F,\nm_F)$ be ultrametric normed vector spaces over $(K,\lvert\cdot\rvert_0)$. The tensor product $\overline E\otimes\overline F$ is defined by equipping $E\otimes F$ with the ultrametric norm
$$\lVert x\rVert_{E\otimes F}:=\min\left\{\max_i\big\{\lVert s_i\rVert_E\cdot\lVert t_i\rVert_F\big\}\big|\text{ }x=\sum_i s_i\otimes t_i\right\} $$
\end{defi}
\subsection{Filtered graded linear series}
\begin{defi}
Let $E_\bullet=\{E_n\}_{n\in\N}$ be a collection of vector subspaces of $K(X)$ over $K$. We say $E_\bullet$ is a \textit{graded linear series} if $\mathop\oplus_{n\in\N}E_n Y^n$ is a graded sub-$K$-algebra of $K(X)[Y]$. If $\mathop\oplus_{n\in\N}E_n Y^n$ is finitely generated $K$-algebra, we say $E_\bullet$ is of finite type. We say $E_\bullet$ is of subfinite type if $E_\bullet$ is contained in a graded linear series. Moreover, if we equip each $E_n$ with an $\R$-filtration $\mathcal F_n^t E_n$, then we call $E_\bullet$ a \textit{filtered graded linear series}.

Denote by $\delta:\N\rightarrow \R$ the function maps $n$ to $C\ln \dim_K(E_n)$. We say $\{\mathcal F_n^t E_n\}_{n\in\N,t\in\R}$ satisfies $\delta$\textit{-superadditivity} if
$$\mathcal F_n^{t_1} E_n\mathcal F_m^{t_2} E_m \subset \mathcal F_{n+m}^{t_1+t_2-\delta(n)-\delta(m)}$$
holds for any $n,m\in \N$ and $t_1,t_2\in\R$
Moreover, we say $\{\mathcal F_n^t E_n\}_{n\in\N,t\in\R}$ is \textit{strong $\delta$-superadditive} if 
$$\mathcal F_{n_1}^{t_1}E_{n_1}\mathcal F_{n_2}^{t_2}E_{n_2}\cdots\mathcal F_{n_r}^{t_r}E_{n_r}\subset \mathcal F_{n_1+n_2+\cdots+n_r}^{t_1+t_2+\cdots+t_r-\delta(n_1)-\delta(n_2)-\cdots-\delta(n_r)}E_{n_1+n_2+\cdots+n_r}$$
holds for any $r\geqslant 2$, $i=1,\dots,r$, $n_i\in \N$ and $t_i\in\R$.
\end{defi}
Since each $\R$-filtration on $E_n$ corresponds to an ultrametric norm $\nm_n$ on $E_n$. The collection $\overline E_\bullet={(E_n,\nm_n)}$ is called an ultrametrically normed graded linear series. The strong $\delta$-superadditivity gives the following inequality:
$$\prod_{i=1}^r\lVert s_i\rVert_{n_i}\leqslant \left\lVert \prod_{i=1}^r s_i\right\rVert_{n_1+n_2+\cdots+n_r}\prod_{i=1}^r\dim_K(E_{n_i})^C$$
holds for any $r\geqslant 2$, $i=1,\dots,r$, $n_i\geqslant 0$ and $s_i\in E_{n_i}$. 

\begin{defi}\label{def_asymp}
We define the volume of a graded linear series $E_\bullet$ of Kodaira-dimension $d$ by
$$\mathrm{vol}(E_\bullet):=\limsup_{n\rightarrow +\infty}\frac{\dim_K(E_n)}{n^d}.$$
If $\overline E_\bullet$ is a ultrametrically normed graded linear series satisfying $\delta$-superadditivity, then we define its
\textit{arithmetic volume} and \textit{arithmetic }$\chi$\textit{-volume} by
$$\begin{aligned}\vol(\overline E_\bullet):=\limsup_{n\rightarrow +\infty} \frac{\sum \max(\widehat{\mu}_i(\overline E_n),0)}{n^{d+1}/(d+1)!},\\
\vol_\chi(\overline E_\bullet):=\limsup_{n\rightarrow +\infty}\frac{\sum \widehat{\mu}_i(\overline E_n)}{n^{d+1}/(d+1)!}.
\end{aligned}$$
Its \textit{asymptotic maximal slope}, \textit{lower asymptotic minimal slope}, and \textit{lower asymptotic minimal slope} is defined respectively by
$$\begin{aligned}\mumax^{asy}(\overline E_\bullet):=\limsup \frac{\mumax(\overline E_n)}{n},\\
\mumin^{\inf}(\overline E_\bullet):=\liminf \frac{\mumin(\overline E_n)}{n},\\
\mumin^{\inf}(\overline E_\bullet):=\limsup \frac{\mumin(\overline E_n)}{n}
\end{aligned}$$
\end{defi}
\section{Reminders on Arakelov geometry over adelic curves}
\subsection{Adelic curves}
\begin{defi}
Let $K$ be a field and $M_K$ be all its places. An \textit{adelic curve} is a 3-tuple $S=\adlcv$ where $(\Omega, \mathcal A, \nu)$ is a measure space consisting of the space $\Omega$, the $\sigma$-algebra $\mathcal A$ and the measure $\nu$, and $\phi$ is a function $(\omega\in\Omega)\mapsto \lvert\cdot\rvert_\omega\in M_K$ such that
$$\omega\mapsto \ln\lvert s\rvert_\omega$$
is $\nu$-integrable for any non-zero element $s$ of $K$. Moreover, if the integral of above function is always zero, then we say $S$ is \textit{proper}. For each $\omega\in\Omega$, we denote by $K_\omega$ the completion field of $K$ with respect to $\lvert\cdot\rvert_\omega$.
\end{defi}

Some typical examples are number fields, projective curves, polarised varieties. We refer to \cite[Section 3.2]{adelic} for detailed constructions. From now on, we assume $S$ to be proper unless it's specified.
\subsection{Adelic vector bundles}
Let $E$ be a vector space over $K$ of dimension $n$. Let $\xi=\{\nm_{\omega}\}$ be a norm family where each $\nm_\omega$ is a norm on $E_{K_\omega}:=E\otimes_K K_\omega$. We can easily define restriction $\xi_F$ of $\xi$ to a subspace $F$, quotient norm family $\xi_{E\twoheadrightarrow G}$ induced by a surjective homomorphism $E\twoheadrightarrow G$, the dual norm family $\xi^\vee$ on $E^\vee$, exterior power norm family and tensor product norm family. Please check \cite[Section 1.1 and 4.1]{adelic} for details.
\begin{defi}[Adelic vector bundles]
We say a norm family $\xi$ is \textit{upper dominated} if
$$\forall s\in E^*, \tint_\Omega \ln\lVert s\rVert_\omega\nu(d\omega)<+\infty.$$
Moreover, we say $\xi$ is \textit{dominated} if its dual norm $\xi^\vee$ on $E^\vee$ is also upper dominated. We say $\xi$ is \textit{measurable} if for any $s\in E^*$, the function $s\mapsto \lVert s\rVert_\omega$ is $\mathcal A$-measurable. If $\xi$ is both dominated and measurable, we say the pair $\overline E=(E,\xi)$ is an \textit{adelic vector bundle}. Note that the propery of being an adelic vector bundle is well-preserved after taking restriction to subspace, quotient, exterior power and tensor product.
\end{defi}

\begin{defi}[Arakelov degrees and slopes]
Let $\overline E=(E,\xi)$ be an adelic vector bundle. If $E\not=0$, then we define its \textit{Arakelov degree} as
$$\deg(\overline E):=-\int_\omega\ln \lVert s\rVert_{\mathrm{det}\,\xi,\omega} \nu(d\omega)$$
where $s\in \mathrm{det}E\setminus \{0\}$ and $\mathrm{det}\,\xi$ is the determinant norm family on $\mathrm{det}E$ (the highest exterior power). Note that this definition is independent with the choice of $s$ since the adelic curve $S$ is proper. If $E=0$, then by convention we define that $\deg(\overline E)=0$.
Moreover, we define the \textit{positive degree} as
$$\deg_+(\overline E):=\sup_{F\subset E}\deg(F,\xi_F),$$
i.e. the supremum of all its adelic vector sub-bundles' Arakelov degerees.
If $E\not=0$, its \textit{slope} $\widehat{\mu}(\overline E)$ is defined to be the quotient $\deg(\overline E)/\dim_K(E)$.
The \textit{maximal slope} and \textit{minimal slope} is defined respectively as
$$\begin{aligned}
\mumax(\overline E)&:=\begin{cases}\sup\limits_{0\not=F\subset E}\widehat \mu(F,\xi_F), &\text{ if }E\not=0\\
-\infty, &\text{ if }E=0 \end{cases}\\
\mumin(\overline E)&:=\begin{cases}\inf\limits_{E\twoheadrightarrow G\not=0}\widehat \mu(G,\xi_{E\twoheadrightarrow G}), &\text{ if }E\not=0\\
+\infty, &\text{ if }E=0 \end{cases}
\end{aligned}
$$
\end{defi}
\begin{rema}
If the field $K$ is of characteristic zero, then there exists a constant $C>0$  such that for any two adelic vector bundles $\overline E$ and $\overline F$, it holds that
$$\mumin(\overline E\otimes \overline F)\geq \mumin(E)+\mumin(F)-C\ln(\dim_K(E)\dim_K(F))$$
which is firstly proved in \cite{BOST2013} and reformulated in \cite[Chapter 5]{adelic}, called \textit{minimal slope property of level}$\geq C$. All the constant $C$ showed up in this article refers to the constant in this sense.
\end{rema}
\begin{defi}
Let $\overline E=(E,\xi)$ be an adelic vector bundle of dimension $n$. The \textit{Harder-Narasimhan} $\R$-filtration is given by
$$\mathcal F_{hn}^t(\overline E)=\sum_{\substack{0\not=F\subset E\\\mumin(F,\xi_F)\geq t}}F.$$
We denote by $\mumin(\overline E)=\widehat\mu_{n}\le \widehat\mu_{n-1}\le\cdots\le \widehat\mu_1=\mumax(\overline E)$ the jumping points of the $\R$-filtration, then it holds that 
$$\begin{aligned}\sum \widehat\mu_i&\le \deg(\overline E)\le \sum \widehat\mu_i+\frac{1}{2}n\ln n,\\
\sum \max(\widehat\mu_i,0)&\le \deg_+(\overline E)\le \sum \max(\widehat\mu_i,0)+\frac{1}{2}n\ln n.
\end{aligned}$$
\end{defi}
\subsection{Adelic Cartier divisors}
Let $X$ be a geometrically irreducible normal projective variety. Let $H$ be an very ample line bundle whose global sections $E:=H^0(X,H)$ is equipped with a dominated norm family $\xi=\{\nm_{\omega}\}_{\omega\in\Omega}$ i.e. $\overline E=(E,\xi)$ is an adelic vector bundle. For each place $\omega\in\Omega$, we denote by $\xanomega$ the analytification of $X\times_{\mathrm{Spec}K}\mathrm{Spec} K_\omega$ in the sense of Berkovich(see \cite{Berkovich}). As in \cite[2.2.3]{adelic}, the norm $\nm_{\omega}$ induces a \textit{Fubini-Study metric} $\varphi_{FS, \overline E,\omega}=\{\lvert\cdot\rvert_{\omega}(x)\}$, which is continuous in the sense that for any open subset $U\subset X$ and a section $s\in H^0(X,U)$, the function 
$$x\in U^{an}_\omega\mapsto \lvert s\rvert_\omega(x)$$
is continuous with respect to Berkovich topology, where $U^{an}_\omega$ is the analytification of $U\times_{\mathrm{Spec}K}\mathrm{Spec} K_\omega$.

For arbitrary line bundle $L$, and two continuous metric families $\varphi=\{\varphi_\omega\}$, $\varphi'=\{\varphi'_\omega\}$, we define the distance function $$\operatorname{dist}(\varphi,\varphi'):(\omega\in\Omega)\mapsto\sup_{x\in \xanomega}\lvert1\rvert_{\varphi_{\omega}-\varphi'_{\omega}}(x)$$
where $\lvert1\rvert_{\varphi_{\omega}-\varphi'_{\omega}}(x)$ is a continuous function on $X$ since $\varphi_\omega-\varphi'_\omega$ is a continous metric of $\O_X$.

We say the pair $(H,\varphi)$ of a very ample line bundle and a continuous metric family is an \textit{adelic very ample line bundle} if $\varphi$ is measurable (see \cite[ 6.1.4]{adelic}), and there exists an dominated norm family $\xi$ on the global sections $E=H^0(X,H)$ such that the distance function
$\operatorname{dist}(\varphi,\varphi_{FS,\overline E})$ is $\nu$-dominated.

As the very ample line bundles generates the Picard group $\mathrm{Pic}(X)$, we define the \textit{arithmetic Picard group} $\widehat{\mathrm{Pic}}(X)$ be the abelian group generated by all adelic very ample line bundles. An element in $\widehat{\mathrm{Pic}}(X)$ is called an \textit{adelic line bundle}.

For a Cartier divisor $D$ on $X$, we define the \textit{$D_\omega$-Green function} $g_\omega$ to be an element of the set
$$C_{gen}^0(\xanomega):=\{f\text{ is a continuous function on }U\mid \emptyset\not= U\mathop{\subset}_{open} \xanomega\}/\sim$$
where $f\sim g\text{ if they are identical on some } V_{\omega}^{an}$ for some non-empty open subset $V\subset X$ such that for any local equation $f_D$ of $D$ on $U$, $\ln|f_D|+g_\omega$ is continuous on $U^{an}_\omega$. It's easy to see that each $D_\omega$-Green function induce a continuous metric on the corresponding line bundle (see \cite[Section 2.5]{adelic} for details). Moreover, we say a pair $(D,g=\{g_\omega\})$ of Cartier divisor $D$ and Green function family $g$ is an \textit{adelic Cartier divisor} if the it corresponds to an adelic line bundle. We denote by $\widehat{Div}(X)$ the group of all adelic Cartier divisor. Let $\mathbb{K}=\Q\textit{ or }\R$, then we can define the set of adelic $\mathbb{K}$-Cartier divisors as
$$\widehat{Div}_{\mathbb{K}}(X)=\widehat{Div}(X)\otimes_\Z \mathbb{K}/\sim$$
where "$\sim$" is the equivalence relationship generated by
$\sum (0,g_i)\otimes k_i\sim (0,\sum g_i k_i)$, where $g_i$'s are continuous funciton families and $k_i\in \mathbb{K}$.

\section{Applications on $\chi$-volume function over an adelic curve}
In this section, we let $S=\adlcv$ be a proper adelic curve where $K$ is of characteristic $0$. Let $X$ be a geometrically irreducible smooth $K$-variety, and $(D,g)$ be a adelic $\R$-Cartier divisor on $X$. Then each $E_n:=H^0(X,\O_X(\lfloor nD\rfloor))$ admits with an adelic vector bundle structure by the norm family $\xi_{ng}$. Then we can equip $E_n$ with the Harder-Narasimhan filtration $\mathcal F^t_n E_n$ induced by $\xi_{ng}$.  Then this filtered linear series satisfies the $\delta$-superadditivity. Notice that there is a correspondence between $\R$-filtrations and ultrametric norms on $E_n$ over trivially valued $K$. We denote by $\lVert\cdot\rVert_n$ the norm induced by $\mathcal F^t_n E_n$. The volume, $\chi$-volume, asymptotic maximal slope. lower asymptotic minimal slope, upper asymptotic minimal slope of $(D,g)$ is defined to be the same with the definition \ref{def_asymp} of the corresponding ultrametrically normed graded linear series $\overline E_\bullet=\{(E_n,\nm_n)\}$.

We recall some previous concerning volume and $\chi$-volume here.
\begin{enumerate}
    \item $\vol(\cdot)$ is continuous on any finitely dimensional subspace of $\widehat{\Div}_\R(X)$.
    \item If $\mumin^{\sup}(D,g)>-+\infty$, then there exists a integrable function $\psi$ on $\Omega$ such that $$\deg(E_n,\xi_{n(g+\psi)})=\deg_+(E_n,\xi_{n(g+\psi)})$$
    for every $n\in \N_+$. In particular, $\vol(D,g+\psi)=\vol_\chi(D,g+\psi)$. 
\end{enumerate}
\subsection{Boundedness of asymptotic minimal slope}
\begin{prop}\label{prop_fin_gen_bound}
Let $\overline E_\bullet=\{\overline E_n=(E_n,\nm_n)\}_{n\in\N}$ be an ultrametrically normed graded linear series of finite type satisfying $\delta$-superadditivity. Then it holds that
$$\mumin^{\inf}(E_\bullet):=\liminf_{n\rightarrow +\infty}\frac{\mumin(\overline E_n)}{n}>-\infty$$
\end{prop}
\begin{proof}
We fix a set of generators, and assume that they are of degree at most $N$.
Then for any $n\in\N$, let $$\Psi_n:=\{(\lambda_1,\lambda_2,\cdots,\lambda_r)\mid 1\leqslant\lambda_1\leqslant\lambda_2\leqslant\cdots\leqslant\lambda_r\leqslant N, \sum_i\lambda_i=n, r=1,\cdots, n\}.$$
Then the map $$\bigoplus^\perp_{(\lambda_1,\lambda_2,\cdots,\lambda_r)\in\Psi_n}(\otimes_{i=1}^r (E_{\lambda_i},\nm_{\lambda_i}\dim_K(E_n)^C)\longrightarrow \overline E_n$$
is surjective and of operator norm$\leqslant 1$. Therefore we have $$\mumin(\overline E_n)\geqslant \min\limits_{(\lambda_1,\lambda_2,\cdots,\lambda_r)\in\Psi_n}\{\sum_{i=1}^r (\mumin(\overline E_{\lambda_i})-\delta(\lambda_i))\}.$$
Let $L=\min\left\{\displaystyle\frac{\mumin(\overline E_l)-\delta(l)}{l}\right\}_{1\leqslant l\leqslant N}.$
Then $\displaystyle\frac{\mumin(\overline E_n)}{n}\geqslant L$ for every $n\in \N$.
\end{proof}

\begin{coro}
If the section ring of $D$ is finitely generated, then for any Green function family $g$ on $D$, it always holds that
$$\mumin^{\inf}(D,g)>-\infty.$$
In particular, the above inequality holds if $D$ is semiample.
\end{coro}

\subsection{Continuity of $\chi$-volume}

\begin{prop}
Let $(L_1,\varphi_1)$, $(L_2,\varphi_2)$, $\dots$, $(L_r, \varphi_r)$ be adelic line bundles on $X$ such that all $L_i$'s are semiample.
For any $a=(a_1,a_2,\dots,a_r)\in \N^r$, denote that $$a\cdot\textnormal{\textbf{L}}:=\sum_{i=1}^r a_i L_i\text{ and }a\cdot \varphi:=\sum_{i=1}^r a_i \varphi_i.$$
Then there exists constants $S$ and $T$ such that
$$\mumin(H^0(X,a\cdot\textnormal{\textbf{L}}),\xi_{a\cdot \varphi})\ge S\cdot |a|+T$$
where $|a|:=a_1+a_2+\cdots+a_r$.
\end{prop}
\begin{proof}
Let $E=\bigoplus_{i=1}^r L_i$, $\mathbb{P}(E):=\mathrm{\mathcal {P}roj_X}(\mathrm{Sym}(E))$ the projective bundle on $X$ associated with $E$. Let $H=\mathcal O_{\mathbb{P}(E)}(1)$.
Then $$H^0(\mathbb{P}(E),mH)=\bigoplus\limits_{a_1+\cdots+a_r=m} H^0(X,\sum_{i=1}^r a_i L_i)$$
Note that since $L_i$'s are semiample, so is $H$ on $\mathbb{P}(E)$.
For each $a\in \N^r$, we denote by $\nm_{(a)}$ the norm on $H^0(X,a\cdot \textbf{L})$ over trivially valued $K$ induced by the Harder-Narasimhan filtration of $\xi_{a\cdot \varphi}$. If we write $a=a^{(1)}+a^{(2)}+\cdots+a^{(l)}$ where $a^{(j)}\in \N^r$, then it holds that for any $u=u_1\cdots u_l\in H^0(X, a\cdot \textbf L)$ where $u_j\in H^0(X, a^{(j)}\cdot\textbf{L})$,
$$\lVert u\rVert_{(a)}\leqslant \prod_{i=1}^l\lVert u_j\rVert_{(a^{(j)})}\dim_K(H^0(X,a^{(j)}\cdot\textbf{L}))^C$$

Let $\nm_m$ be the direct sum of $\{\nm_{(a)}\mid a\in \N^r,|a|=m\}$.
Then it suffices to prove that
the normed graded algebra $$\bigoplus_{m\geqslant 0}(H^0(\mathbb{P}(E),mH),\nm_m)$$ is strong $\delta$-supperadditive. 
Let $s_i=\sum\limits_{|a|=m_i} u_a^{(i)}\in H^0(\mathbb{P}(E),m_iH)$ where $i=1,2,\dots,l$, $m_i\in \N$, $u_{a}^{(i)}\in H^0(X,a\cdot \textbf{L})$.
It holds that
$$\lVert s_i\rVert_{m_i}=\max_{|a|=m_i}\{\lVert u_a^{(i)}\rVert_{(a)}\}.$$
Moreover,
$$\allowdisplaybreaks{\begin{aligned}\lVert s_1\cdots s_l\rVert_m&=\max\limits_{|\alpha|=m}\left\{\left\lVert \sum\limits_{\substack{a^{(1)}+\cdots+a^{(l)}=a\\ |a^{(i)}|=m_i}}\prod_{i=1}^l u_{a^{(i)}}^{(i)}\right\rVert_{(\alpha)}\right\}\\
&\leqslant\max_{\substack{|a^{(i)}|=m_i\\i=1,\dots,l}}\left\lVert \prod_{i=1}^l u_{a^{(i)}}^{(i)}\right\rVert_{(a^{(1)}+a^{(2)}+\cdots+a^{(l)})}\\
&\leqslant \max_{\substack{|a^{(i)}|=m_i\\i=1,\dots,l}}\prod_{i=1}^l\left\lVert u_{a^{(i)}}^{(i)}\right\rVert_{(a^{(i)})}\dim_K(H^0(X,a^{(i)}\cdot \textbf{L}))^C\\
&\leqslant \prod_{i=1}^l{\max_{|a|=m_i}\{\lVert u_{a}^{(i)}\rVert_{(a)}\dim_K(H^0(X,a\cdot \textbf{L}))^C}\}\\
&\leqslant \prod_{i=1}^l\lVert s_i\rVert_{m_i}\dim_K(H^0(\mathbb{P}(E),m_iH))^C.
\end{aligned}}$$
Therefore there exists constants $S$ and $T$ such that $$\mumin(H^0(X,a\cdot L),\xi_{a\cdot \varphi})\geq \mumin(H^0(\mathbb{P}(E),|a|H),\nm_{|a|})\ge S\cdot |a|+T.$$
\end{proof}

\begin{theo}
Let $\overline D=(D,g)$, $\overline E_1=(E_1,h_1)$, $\dots, \overline E_r=(E_r,h_r)$ be adelic $\Q$-Cartier divisors on $X$ such that $D$ and $E_i$'s are semiample.
\begin{equation}\label{eq_conti_semiample}\lim_{\substack{\epsilon_1+\cdots+\epsilon_r\rightarrow 0\\\epsilon_i\in\Q_{\geq 0}}}\vol_\chi(\overline D+\sum_{i=1}^r\epsilon_i \overline E_i)=\vol_\chi(\overline D)\end{equation}
\end{theo}
\begin{proof}
Due to the homogeneity, we may assume that all $D$ and $E_i$ are integral semiample Cartier divisors. 
Then there exists constants $S$ and $T$ depending on $\overline D$ and $\overline E_1,\dots,\overline E_r$, such that
$$\mumin(H^0(X,n_0 D+\sum_{i=1}^r n_i E_i),\xi_{n_0g+\sum_{i=1}^r n_i h_i})\geqslant T+S\sum_{i=0}^{r}n_i$$
where $n_i\in \N$
Assume that $\epsilon_i=p_i/q_i$ where $p_i$ and $q_i$ are coprime positive integers and $q_i\geqslant 1$. Let $q=\prod_{i=1}^r q_i$, then it holds that
$$\mumin(H^0(X,mq(D+\sum_{i=1}^r \epsilon_i E_i)),\xi_{nq(g+\sum_{i=1}^r \epsilon_i h_i)})\geqslant T+Smq(1+\sum_{i=1}^r \epsilon_i)$$
for every $m\in\N$.
Therefore
$$\mumin^{\sup}(\overline D+\sum_{i=1}^r \epsilon_i \overline E_i)\geqslant S(1+\sum_{i=1}^r\epsilon_i).$$
Take $\nu$-integrable functions $\phi$ such that
$\displaystyle\int_\Omega \phi \nu(d\omega) =S$. Denote that $|\epsilon|=\sum_{i=1}^r \epsilon_i$, it holds that
\begin{equation}\label{eq_vol_chi_semiample}\vol(\overline D+\sum_{i=1}^r \epsilon_i \overline E_i+(0,1+|\epsilon|\phi))=\vol_\chi(\overline D+\sum_{i=1}^r \epsilon_i \overline E_i+(0,(1+|\epsilon|)\phi)).\end{equation}
Due to the continuity of $\vol(\cdot)$, (\ref{eq_conti_semiample}) can be easily derived from (\ref{eq_vol_chi_semiample}).
\end{proof}

\section{Applications on arithmetic Hilbert-Samuel function}
\subsection{Asymptotically modified norm}
Let $\overline E_\bullet=\{\overline E_n=(E_n,\nm_n)\}_{n\in\N}$ be an ultrametrically normed graded linear series of finite type satisfying $\delta$-superadditivity. Let $r_n=\dim_K(E_n)$.
\begin{prop}\label{prop_asy_norm}For any $n\in\N$ and $s\in E_n$, we define
$$\lVert s\rVert'_n:=\liminf_{m\rightarrow \infty}\lVert s^m\rVert_{nm}^{\frac{1}{m}}$$
Then the following four properties holds:
\begin{enumerate}
    \item[\textnormal{(1)}] $\lVert s\rVert'_n=\lim\limits_{m\rightarrow \infty}\lVert s^m\rVert_{nm}^{\frac{1}{m}}$.
    \item[\textnormal{(2)}] $\lVert\cdot\rVert'_n$ is an ultrametric norm on $E_n$ over trivially valued $K$.
    \item[\textnormal{(3)}] $\lVert s\rVert'_n \lVert t\rVert'_m\geqslant \lVert st\rVert'_{n+m}$ for any $s\in E_n$ and $t\in E_m$.
    \item[\textnormal{(4)}] $\displaystyle\frac{\lVert s\rVert'_n}{\lVert s\rVert_n}\leqslant r_n^C .$
\end{enumerate}
\end{prop}
\begin{proof}
(1)Let $a_m:=-\ln\|s^m\|_{nm}$. Since $\mathcal F_n^t$ is $\delta$-superadditive, we have
$$\mathcal F^{a_m}_{nm}E_{nm}\mathcal F^{a_{m'}}_{nm'}E_{nm'}\subset \mathcal F^{a_m+a_{m'}-\delta(nm)-\delta(nm')}E_{n(m+m')}$$
which implies that $a_{m+m'}\geqslant a_m+a_{m'}-\delta(nm)-\delta(nm')$. So the sequence $\left\{\displaystyle\frac{a_m}{m}\right\}$ converges in $\R$. 

(2) For any $s,t\in E_n$, $\lVert(s+t)^m\rVert_{nm}\leqslant \max\limits_{i=0,\dots,m}\left\{\lVert s^it^{m-i}\rVert_{nm} \right\}$.
Since the filtration is $\delta$-supperadditive, we can deduce that $$\lVert s^i t^{m-i}\rVert_{nm}\leqslant \lVert s^i\rVert_{ni}\lVert t^{m-i}\rVert_{n(m-i)}\exp(\delta(ni)+\delta(n(m-i))).$$
Let $A=\max\{\lVert s\rVert'_n,\lVert t\rVert'_n\}$. Then for any $\epsilon>0$, there is an integer $N$ such that $\lVert s^l\rVert_{nl}^{\frac{1}{l}}\exp(\frac{\delta(nl)}{l})\leqslant A+\epsilon$ and $\lVert t^l\rVert_{nl}^{\frac{1}{l}}\exp(\frac{\delta(nl)}{l})\leqslant A+\epsilon$ for every $l>N$.

Let $B=\max\limits_{i=0,\dots,N}\{\max(\lVert s^i\rVert_{ni}, \lVert t^i\rVert_{ni})\exp(\delta(ni))\}$.
When $m>2N$, either $i$ or $m-i$ is greater than $N$, thus
$$\lVert (s+t)^m\rVert_{nm}^{\frac{1}{m}}\leqslant \max\limits_{i=0,\dots,N}\{(A+\epsilon)^{\frac{m-i}{m}}B^{\frac{1}{m}}\}.$$
The right hand side of the inequality has the limit $A+\epsilon$ which implies that there exists $N'\in\N_+$ such that 
$\lVert(s+t)^m\rVert_{nm}^{\frac{1}{m}}\leqslant A+2\epsilon$ for every $m>N'$.
Therefore $\lVert s+t\rVert'_n\leqslant A$.

(3) Due to the $\delta$-superadditivity, we have $$\frac{-\ln\lVert(st)^l\rVert_{l(n+m)}}{l}\geqslant \frac{-\ln\lVert s^l\rVert_{nl}}{l}+\frac{-\ln\lVert t^l\rVert_{ml}}{l}-\frac{\delta(nl)}{l}-\frac{\delta(ml)}{l}$$ for every $l\in\N_+$. Let $l\rightarrow +\infty$, we obtain (3).

(4) This is a direct result from an estimate of the limit.
\end{proof}

By (2) of \ref{prop_asy_norm}, we can give an $\R$-filtration of $E_n$ induced by $\lVert\cdot\rVert'_n$. Let $\overline E'_n:=(E_n,\nm'_n)$. It's easy to see that $\widehat\mu_i(\overline E'_n)+\delta(n)\geqslant \widehat\mu_i(\overline E_n)$ because of (4) of Proposition \ref{prop_asy_norm}.
\begin{prop} It holds that
     \begin{equation}\label{eq_mumax}
         \mumax^{asy}(\overline E_\bullet)=\mumax^{asy}(\overline E'_\bullet).
     \end{equation}
\end{prop}
\begin{proof}
It's obvious that $\mumax^{asy}(\overline E_\bullet)\leq \mumax^{asy}(\overline E'_\bullet)$. Conversely, for any $t<\mumax(\overline E'_n)$, there exists an increasing sequence of integers $\{n_k\in\N\}_{k\in\N}$ and a sequence of sections $\{s_k\in E_{n_k}\}_{k\in\N_+}$, such that for all $k\in\N$
$$\frac{-\ln\lVert s_k\rVert'_{n_k}}{n_k}> t.$$
By the definition of $\nm'_n$, for each $k\in \N$, $-\ln\lVert s_k^m\rVert_{mn_k}^{\frac{1}{mn_k}}>t$ for every $m\gg 0$. Therefore we can obtain another increasing sequence of integers $\{l_k:=m_k n_k|m_k\in \N_+\}$ and $t_k:=s_k^{m_k}$ such that
$$-\ln\lVert t_k\rVert_{l_k}^{\frac{1}{l_k}}>t$$
which implies that $\mumax^{asy}(\overline E_\bullet)\geq t$. Hence the equation (\ref{eq_mumax}) is obtained.
\end{proof}

For each $t\in\R$, we define $E_n^t:=\{s\in E_n\mid -\ln\lVert s\rVert'_n\geqslant nt\}$ and $E^t_\bullet:=\{E_n^t\}_{n\in\N}$. Then $E^t_\bullet$ is a graded linear series of subfinite type due to property (3) of Proposition \ref{prop_asy_norm}.

\subsection{Estimate on graded series of subfinite type}

This part is a result of \cite{chen:hilbertsamuel} after some modifications.
Let $X$ be a projective $K$-scheme of dimension $d$ admitting the following collection of morphisms $\{p_i:X_i\rightarrow C_i\}_{i=1,\dots,n}$
given by following construction:
\begin{enumerate}
    \item If $d=1$, then $C_1:=X$, $X_1\rightarrow C_1$ is the normalization of $X$.
    \item If $d>1$, then $C_1$ is a projective regular curve over $\mathrm{Spec}K$, $X_1=X$ and $p_1$ is a projective and flat $k$-morphism.
    \item For any $i\in\{2,\dots, d-1\}$, $C_i$ is a projective regular curve over $K(C_{i-1})$. $X_i$ is the the generic fiber of $p_{i-1}$ and $p_i$ is a projective flat morphism of $K(C_{i-1})$-schemes.
    \item $X_d$ is the normalization of the generic fiber of $p_{d-1}$.
\end{enumerate}

\begin{defi}
Let $L$ be a big line bundle on $X$. Let $E_\bullet$ be a graded subalgebra of the section ring $R(X,L)$ of $L$. We say that $E_\bullet$ \textit{contains an ample divisor} if
\begin{enumerate}
    \item There exists an $m\gg 0$ and a decomposition
    $$mL=A+F$$ where $A$ is ample and $F$ is effective.
    \item For every $k\gg 0$, we have the inclusion
    $$H^0(X,kA_m)\subset E_{km}\subset H^0(X,kmL)$$
\end{enumerate}
\end{defi}

\begin{lemm}
Let $(L,\varphi)$ be an adelic Cartier divisor on $X$ with $L$ being big. Let $\overline E_\bullet$ be the corresponding ultrametrically normed graded linear series. Then for any $t\leqslant\mumax^{asy}(L,\varphi)$, the graded linear series $E_\bullet^t$ given by asymptotically modified norms as described in previous subsection contains an ample divisor.
\end{lemm}
\begin{proof}
See \cite[Lemma 1.6]{BC_Okounkov}.
\end{proof}

\begin{theo}\label{hil_sam_lin_ser}
Let $L$ be a big line bundle on $X$. There exists a $f(n)\simeq O(n^{d-1})$ such that for any graded subalgebra $F_\bullet$ of the section ring $R(X,L)$ containing an ample divisor, it holds that
$$\dim_K(F_n)\leqslant \mathrm{vol}(F_\bullet)n^d+f(n).$$
\end{theo}
\begin{proof}
See \cite[Theorem 5.1]{chen:hilbertsamuel}.
\end{proof}

\begin{theo}[Upper bound of arithmetic Hilbert-Samuel function]\label{Hil} Let $(D,g)$ be an adelic divisor on $X$ with $D$ being big. For each $n\in \N$, let $\overline E_n$ denote the pair of $E_n=H^0(X,\O_X(nD))$ and norm family $\xi_{ng}$. Then it holds that
$$\deg_+(\overline E_n)\leqslant \vol(D,g)\frac{n^{d+1}}{(d+1)!}+O(n^d)+(C+1/2)\cdot r_n \ln(r_n).$$
\end{theo}
\begin{proof}
Let $\nm_n$ be the ultrametric norm induced by the Harder-Narasimhan filtration on $\overline E_n$.
We apply the asymptotic modification on $\{(E_n,\nm_n)\}$
$$\begin{aligned}\deg_+(\overline E_n)&\leqslant \sum_{i=1}^{r_n}\max(\widehat\mu_i(\overline E_n),0)+1/2\cdot r_n \ln(r_n)\\
&\leqslant \sum_{i=1}^{r_n}\max(\widehat\mu'_i(\overline E_n),0)+(C+1/2)\cdot r_n \ln(r_n)\\
&= n\int_0^{+\infty}\dim_K(E_n^t)dt+C\cdot r_n \ln(r_n)+\delta(\overline E_n)\\
&\leqslant n\int_0^{\mumax^{asy}(D,g)}\mathrm{vol}(E^t_\bullet){n^d}dt+O(n^d)+(C+1/2)\cdot r_n \ln(r_n)\\
&=\vol(D,g)\frac{n^{d+1}}{(d+1)!}+O(n^d)+(C+1/2)\cdot r_n \ln(r_n)
\end{aligned}$$
The last equation is obtained by the virtue of \cite[Theorem 6.3.16]{adelic}. 
\end{proof}

\begin{coro} We keep the hypothesis in Theorem \ref{Hil}. If $D$ is semi-ample, then 
\begin{equation}\label{eq_hilbert}\deg(\overline E_n)\leqslant \vol_\chi(D,g)\frac{n^{d+1}}{(d+1)!}+O(n^d)+(C+1/2)\cdot r_n\ln(r_n)\end{equation}
\end{coro}
\begin{proof}
Since there exists an integrable function $\psi$ on $\Omega$ such that 
$$\forall n\in\N_+, \deg(E_n,\xi_{ng}+n\psi)=\deg_+(E_n,\xi_{ng}+\psi)$$
and $$\vol(D,g+\psi)=\vol_\chi(D,g+\psi),$$
applying previous proposition, we obtain that
$$\begin{aligned}\deg(&\overline E_n)+n\cdot r_n A\\ &\leqslant \vol_\chi(D,g) \frac{n^{d+1}}{(d+1)!}+\mathrm{vol}(D) \frac{n^{d+1}}{d!}A+O(n^d)+(C+1/2)\cdot r_n \ln(r_n)\end{aligned}$$
where $A=\int_\Omega \psi \nu(d\omega)$.
Then the formula (\ref{eq_hilbert}) is obtained by Riemann-Roch type inequality of Kollár and Matsusaka\cite{kollar1983}.
\end{proof}

\bibliography{mybibliography}
\bibliographystyle{smfplain}
\end{document}